\documentclass[reqno,12pt]{amsart}
\usepackage{amsmath, epsfig, cite}
\usepackage{amssymb}
\usepackage{amsfonts}
\usepackage{latexsym}
\usepackage{amsthm}

\newtheorem{theorem}{Theorem}%[section]

\newtheorem{conjecture}{Conjecture}
\newtheorem{lemma}{Lemma}

\theoremstyle{remark}

\numberwithin{equation}{section}

\allowdisplaybreaks[1]

\author{Victor J.\ W.\ Guo}
\address{School of Mathematics and Statistics, Huaiyin Normal University,
Huai'an 223300, Jiangsu, People's Republic of China}
\email{jwguo@hytc.edu.cn}
\thanks{The first author was partially supported by the National Natural
Science Foundation of China (grant 11771175).}

\author{Michael J.\ Schlosser}
\address{Fakult\"at f\"ur Mathematik, Universit\"at Wien,
Oskar-Morgenstern-Platz~1, A-1090 Vienna, Austria}
\email{michael.schlosser@univie.ac.at}
\thanks{The second author was partially supported by FWF Austrian Science Fund
grant P 32305.}
\title[Some $q$-supercongruences modulo $\Phi_n(q)^2$ and  $\Phi_n(q)^3$]
{Some $q$-supercongruences modulo the square and cube
  of a cyclotomic polynomial}

\subjclass[2010]{Primary 33D15; Secondary 11A07, 11B65}

\keywords{basic hypergeometric series; supercongruences; $q$-congruences;
cyclotomic polynomial; Andrews' transformation; Gasper's summation.}

\begin{document}

\begin{abstract}
Two $q$-supercongruences of truncated basic hypergeometric series
containing two free parameters are established by employing specific identities
for basic hypergeometric series. The results partly extend two
$q$-supercongruences that were earlier conjectured by the same authors
and involve $q$-supercongruences modulo the square and the cube of a
cyclotomic polynomial. One of the newly proved $q$-supercongruences
is even conjectured to hold modulo the fourth power of a cyclotomic polynomial.
\end{abstract}

\maketitle

\section{Introduction}
In 1914, Ramanujan \cite{Ramanujan} listed a number of representations
of $1/\pi$, including
\begin{align}
\sum_{k=0}^{\infty}(6k+1)\frac{(\frac{1}{2})_k^3}{k!^3 4^k}
=\frac{4}{\pi}, \label{eq:ram}
\end{align}
where $(a)_n=a(a+1)\cdots(a+n-1)$ denotes the Pochhammer symbol.
Ramanujan's formulas gained unprecedented popularity in the 1980's
when they were discovered to provide fast algorithms for calculating
decimal digits of $\pi$. See, for instance, the monograph \cite{BB}
by the Borwein brothers.

In 1997, Van Hamme \cite{Hamme} conjectured 13 intriguing $p$-adic
analogues of Ramanujan-type formulas, such as
\begin{align}
\sum_{k=0}^{(p-1)/2}(6k+1)\frac{(\frac{1}{2})_k^3}{k!^3 4^k}
&\equiv p(-1)^{(p-1)/2}\pmod{p^4}, \label{eq:pram}
\end{align}
where  $p>3$ is a prime. Van Hamme himself supplied proofs for
three of them. Supercongruences like \eqref{eq:pram} are called
Ramanujan-type supercongruences (see \cite{Zud2009}).
The proof of the supercongruence \eqref{eq:pram} was first given
by Long~\cite{Long}. As of today, all of Van Hamme's 13
supercongruences have been confirmed by various techniques
(see \cite{OZ,Swisher}).

In recent years, $q$-congruences and $q$-supercongruences
have been established by different authors (see, for example,
\cite{Gorodetsky,Guo-rima,Guo-c2,GS19fifth,GS19,GSnew,GS20,GS1,GS,GuoZu,GuoZu2,LW,Liu,LP,NP2,Straub,WY,WYu,Zu19}).
In particular, the present authors \cite{GS19} proved that,
for any odd integer $d\geqslant 5$,
\begin{equation}
\sum_{k=0}^{n-1}[2dk+1]\frac{(q;q^d)_k^d}{(q^d;q^d)_k^d}q^{d(d-3)k/2}
\equiv
\begin{cases} 0\pmod{\Phi_n(q)^2}, &\text{if $n\equiv -1\pmod{d}$,}\\[5pt]
0\pmod{\Phi_n(q)^3}, &\text{if $n\equiv -1/2\pmod{d}$.}
\end{cases}  \label{eq:old-1}
\end{equation}
Here and in what follows, we adopt the standard $q$-notation:
$[n]=1+q+\cdots+q^{n-1}$ is the {\em $q$-integer};
$(a;q)_n=(1-a)(1-aq)\cdots (1-aq^{n-1})$ is the {\em $q$-shifted factorial},
with the compact notation
$(a_1,a_2,\ldots,a_m;q)_n=(a_1;q)_n (a_2;q)_n\cdots (a_m;q)_n$
used for their products; and $\Phi_n(q)$ denotes the $n$-th
{\em cyclotomic polynomial} in $q$, which may be defined as
\begin{align*}
\Phi_n(q)=\prod_{\substack{1\leqslant k\leqslant n\\ \gcd(k,n)=1}}(q-\zeta^k),
\end{align*}
where $\zeta$ is an $n$-th primitive root of unity.

We should point out that the $q$-congruence \eqref{eq:old-1} does not
hold for $d=3$. The present authors \cite{GS19} also established the following
companion of \eqref{eq:old-1}:
for any odd integer $d\geqslant 3$ and integer $n>1$,
\begin{equation}
\sum_{k=0}^{n-1}[2dk-1]\frac{(q^{-1};q^d)_k^d}{(q^d;q^d)_k^d}
q^{d(d-1)k/2}
\equiv
\begin{cases} 0\pmod{\Phi_n(q)^2}, &\text{if $n\equiv 1\pmod{d}$,}\\[5pt]
0\pmod{\Phi_n(q)^3}, &\text{if $n\equiv 1/2\pmod{d}$.}
\end{cases} \label{eq:old-2}
\end{equation}
They also proposed the following conjectures \cite[Conjectures 1 and 2]{GS19},
which are generalizations of \eqref{eq:old-1} and \eqref{eq:old-2}.
\begin{conjecture}\label{conj-1}
Let $d\geqslant 5$ be an odd integer. Then
\begin{equation*}
\sum_{k=0}^{n-1}[2dk+1]\frac{(q;q^d)_k^d}{(q^d;q^d)_k^d}q^{d(d-3)k/2}
\equiv
\begin{cases} 0\pmod{\Phi_n(q)^3}, &\text{if $n\equiv -1\pmod{d}$,}\\[5pt]
0\pmod{\Phi_n(q)^4}, &\text{if $n\equiv -1/2\pmod{d}$.}
\end{cases}
\end{equation*}
\end{conjecture}

\begin{conjecture}\label{conj-2}
Let $d\geqslant 5$ be an odd integer and let $n>1$. Then
\begin{equation*}
\sum_{k=0}^{n-1}[2dk-1]\frac{(q^{-1};q^d)_k^d}{(q^d;q^d)_k^d}q^{d(d-1)k/2}
\equiv
\begin{cases} 0\pmod{\Phi_n(q)^3}, &\text{if $n\equiv 1\pmod{d}$,}\\[5pt]
0\pmod{\Phi_n(q)^4}, &\text{if $n\equiv 1/2\pmod{d}$.}
\end{cases}
\end{equation*}
\end{conjecture}

$q$-Supercongruences such as those above (modulo
a third and even fourth power of a cyclotomic polnomial) are rather special.
In fact, concrete results for truncated basic hypergeometric sums
being congruent to $0$ modulo a high power of a cyclotomic polynomial
are very rare. See \cite{GS19fifth,GS20,GSnew,GS1,GW,LW}
for recent papers featuring such results.
The main goal of this paper is to add two complete two-parameter families
of $q$-supercongruences to the list of such $q$-supercongruences
(see Theorems~\ref{thm:1} and \ref{thm:2}).

We shall prove that the respective first cases of
Conjectures~\ref{conj-1} and \ref{conj-2} are true
by establishing the following more general result.
\begin{theorem}\label{thm:1}
Let $d$ and $r$ be odd integers satisfying $d\geqslant 5$, $r\leqslant d-4$
(in particular, $r$ may be negative) and $\gcd(d,r)=1$.
Let $n$ be an integer such that $n\geqslant d-r$
and $n\equiv-r\pmod{d}$. Then
\begin{equation}
  \sum_{k=0}^{M}[2dk+r]\frac{(q^r;q^d)_k^d}{(q^d;q^d)_k^d}q^{d(d-r-2)k/2}
  \equiv 0 \pmod{[n]\Phi_{n}(q)^2},
\label{eq:thm1}
\end{equation}
where $M=(dn-n-r)/d$ or $n-1$.
\end{theorem}

We shall also prove the following $q$-supercongruences.
\begin{theorem}\label{thm:2}
Let $d$ and $r$ be odd integers satisfying $d\geqslant 5$, $r\leqslant d-4$
(in particular, $r$ may be negative) and $\gcd(d,r)=1$.
Let $n$ be an integer such that $n\geqslant (d-r)/2$
and $n\equiv-r/2\pmod{d}$. Then
\begin{equation}
  \sum_{k=0}^{M}[2dk+r]\frac{(q^r;q^d)_k^d}{(q^d;q^d)_k^d}q^{d(d-r-2)k/2}
  \equiv 0 \pmod{[n]\Phi_{n}(q)},
\label{eq:thm2}
\end{equation}
where $M=(dn-2n-r)/d$ or $n-1$.
\end{theorem}

The following generalization of the respective second cases
of Conjectures~\ref{conj-1} and \ref{conj-2} should be true.

\begin{conjecture}
  The $q$-supercongruence \eqref{eq:thm2} holds
  modulo $[n]\Phi_{n}(q)^3$.
\end{conjecture}

We shall prove Theorems~\ref{thm:1} and \ref{thm:2}
in Sections~\ref{sec:thm1} and \ref{sec:thm2}, respectively,
by making use of Andrews' multiseries
extension \eqref{andrews} of the Watson
transformation~\cite[Theorem~4]{Andrews75},
along with Gasper's very-well-poised
Karlsson--Minton type summation \cite[Eq.~(5.13)]{Gasper}.
It should be pointed out that Andrews' transformation
plays an important part in combinatorics and number theory
(see \cite{Guo-c2} and the introduction of \cite{GS1} for more such examples).

%%%%%%%%%%%%%%%%%%%%%%%%%%%%%%%%%%%%%%%%%%%%%%%%%%%%%%%%%%%%%%%%%%%%%%%%%%%%%%%%%%%%%%%%%%%%%%%%%%%%%%%%%%%%%%%%%%%%%%%%%%%%%%%%%%%%%%%%%%%%%%%%%%%%%%%%%%%%%%%%%%%%%%%%%%%%%%%%%%%%%%%%%%%%%%%%%%%%%%%%%%%%%%%%%%%%%%%%%%%%%%%%%%%%%
\section{Proof of Theorem \ref{thm:1}}\label{sec:thm1}
We need a simple $q$-congruence modulo $\Phi_n(q)^2$, which was
already used in \cite{GS20,GS1}.
\begin{lemma}\label{lem:mod-square}
Let $\alpha$, $r$ be integers and $n$ a positive integer. Then
\begin{equation}
(q^{r-\alpha  n},q^{r+\alpha  n};q^d)_k \equiv (q^r;q^d)_k^2 \pmod{\Phi_n(q)^2}.
\label{eq:mod-square}
\end{equation}
\end{lemma}

We will further utilize a powerful transformation formula due to
Andrews \cite[Theorem~4]{Andrews75}, which may be stated as follows:
\begin{align}
\sum_{k\geqslant 0}\frac{(a,q\sqrt{a},-q\sqrt{a},b_1,c_1,\dots,b_m,c_m,q^{-N};q)_k}
{(q,\sqrt{a},-\sqrt{a},aq/b_1,aq/c_1,\dots,aq/b_m,aq/c_m,aq^{N+1};q)_k}
\left(\frac{a^mq^{m+N}}{b_1c_1\cdots b_mc_m}\right)^k &\notag\\[5pt]
=\frac{(aq,aq/b_mc_m;q)_N}{(aq/b_m,aq/c_m;q)_N}
\sum_{j_1,\dots,j_{m-1}\geqslant 0}
\frac{(aq/b_1c_1;q)_{j_1}\cdots(aq/b_{m-1}c_{m-1};q)_{j_{m-1}}}
{(q;q)_{j_1}\cdots(q;q)_{j_{m-1}}} \notag\\[5pt]
\times\frac{(b_2,c_2;q)_{j_1}\dots(b_m,c_m;q)_{j_1+\dots+j_{m-1}}}
{(aq/b_1,aq/c_1;q)_{j_1}
\dots(aq/b_{m-1},aq/c_{m-1};q)_{j_1+\dots+j_{m-1}}}& \notag\\[5pt]
\times\frac{(q^{-N};q)_{j_1+\dots+j_{m-1}}}
{(b_mc_mq^{-N}/a;q)_{j_1+\dots+j_{m-1}}}
\frac{(aq)^{j_{m-2}+\dots+(m-2)j_1} q^{j_1+\dots+j_{m-1}}}
{(b_2c_2)^{j_1}\cdots(b_{m-1}c_{m-1})^{j_1+\dots+j_{m-2}}}&.  \label{andrews}
\end{align}
This transformation is a multiseries generalization of Watson's
$_8\phi_7$ transformation formula (listed in \cite[Appendix (III.18)]{GR};
cf.\ \cite[Chapter~1]{GR} for the notation of a basic hypergeometric
$_r\phi_s$ series we are using),
\begin{align}
& _{8}\phi_{7}\!\left[\begin{array}{cccccccc}
a,& qa^{\frac{1}{2}},& -qa^{\frac{1}{2}}, & b,    & c,    & d,    & e,    & q^{-n} \\
  & a^{\frac{1}{2}}, & -a^{\frac{1}{2}},  & aq/b, & aq/c, & aq/d, & aq/e, & aq^{n+1}
\end{array};q,\, \frac{a^2q^{n+2}}{bcde}
\right] \notag\\[5pt]
&\quad =\frac{(aq, aq/de;q)_n}
{(aq/d, aq/e;q)_n}
\,{}_{4}\phi_{3}\!\left[\begin{array}{c}
aq/bc,\ d,\ e,\ q^{-n} \\
aq/b,\, aq/c,\, deq^{-n}/a
\end{array};q,\, q
\right], \label{eq:8phi7}
\end{align}
to which it reduces for $m=2$.

Next, we require a very-well-poised Karlsson--Minton type summation
due to Gasper~\cite[Eq.~(5.13)]{Gasper} (see also \cite[Ex.~2.33 (i)]{GR}):
\begin{align}
\sum_{k=0}^\infty\frac{(a,q\sqrt{a},-q\sqrt{a},b,a/b,d,e_1,aq^{n_1+1}/e_1,\dots,
e_m,aq^{n_m+1}/e_m;q)_k}{(q,\sqrt{a},-\sqrt{a},aq/b,bq,aq/d,aq/e_1,e_1q^{-n_1},
\dots,aq/e_m,e_mq^{-n_m};q)_k}\left(\frac{q^{1-\nu}}d\right)^k&\notag\\
=\frac{(q,aq,aq/bd,bq/d;q)_\infty}{(bq,aq/b,aq/d,q/d;q)_\infty}
\prod_{j=1}^m\frac{(aq/be_j,bq/e_j;q)_{n_j}}{(aq/e_j,q/e_j;q)_{n_j}}&,
\label{eq:gasper}
\end{align}
where $n_1,\dots,n_m$ are non-negative integers,
$\nu=n_1+\cdots+n_m$, and the convergence condition $|q^{1-\nu}/d|<1$
if the series does not terminate.
We point out that an elliptic extension of the terminating $d=q^{-\nu}$ case of
\eqref{eq:gasper} can be found in \cite[Eq.~(1.7)]{RS}.

In particular, we note that for $d=bq$ the right-hand side of
\eqref{eq:gasper} vanishes.
Putting in addition $b=q^{-N}$ we get the following terminating summation formula:
\begin{equation}\label{eq:vwp-km}
\sum_{k=0}^N\frac{(a,q\sqrt{a},-q\sqrt{a},e_1,aq^{n_1+1}/e_1,\dots,
e_m,aq^{n_m+1}/e_m,q^{-N};q)_k}{(q,\sqrt{a},-\sqrt{a},aq/e_1,e_1q^{-n_1},
\dots,aq/e_m,e_mq^{-n_m},aq^{N+1};q)_k}q^{(N-\nu)k}=0,
\end{equation}
which is valid for $N>\nu=n_1+\cdots+n_m$.

A suitable combination of \eqref{andrews} and \eqref{eq:vwp-km}
yields the following multi-series summation formula, derived in
\cite[Lemma~2]{GS1} (whose proof we nevertheless give here,
to make the paper self-contained):
\begin{lemma}\label{lem:ms=0}
Let $m\geqslant 2$. Let $q$, $a$ and $e_1,\dots,e_{m+1}$ be arbitrary
parameters with
$e_{m+1}=e_1$, and let $n_1,\dots,n_m$ and $N$ be non-negative integers
such that $N>n_1+\cdots+n_m$. Then
\begin{align}
0=\sum_{j_1,\dots,j_{m-1}\geqslant 0}
\frac{(e_1q^{-n_1}/e_2;q)_{j_1}\cdots(e_{m-1}q^{-n_{m-1}}/e_m;q)_{j_{m-1}}}
{(q;q)_{j_1}\cdots(q;q)_{j_{m-1}}} \notag\\[5pt]
\times\frac{(aq^{n_2+1}/e_2,e_3;q)_{j_1}\dots
(aq^{n_m+1}/e_m,e_{m+1};q)_{j_1+\dots+j_{m-1}}}
{(e_1q^{-n_1},aq/e_2;q)_{j_1}
\dots(e_{m-1}q^{-n_{m-1}},aq/e_m;q)_{j_1+\dots+j_{m-1}}}& \notag\\[5pt]
\times\frac{(q^{-N};q)_{j_1+\dots+j_{m-1}}}
{(e_1q^{n_m-N+1}/e_m;q)_{j_1+\dots+j_{m-1}}}
\frac{(aq)^{j_{m-2}+\dots+(m-2)j_1} q^{j_1+\dots+j_{m-1}}}
{(aq^{n_2+1}e_3/e_2)^{j_1}\cdots
(aq^{n_{m-1}+1}e_m/e_{m-1})^{j_1+\dots+j_{m-2}}}&. \label{mkm0}
\end{align}
\end{lemma}
\begin{proof}
By specializing the parameters in the multi-sum transformation \eqref{andrews}
by $b_i\mapsto aq^{n_i+1}/e_i$, $c_i\mapsto e_{i+1}$, for $1\le i\le m$
(where $e_{m+1}=e_1$), and dividing both sides of
the identity by the prefactor of the multi-sum,
we obtain that the series on the right-hand side of \eqref{mkm0} equals
\begin{align*}
&\frac{(e_mq^{-n_m},aq/e_1;q)_N}{(aq,e_mq^{-n_m}/e_1;q)_N}\\&\times
\sum_{k=0}^N\frac{(a,q\sqrt{a},-q\sqrt{a},e_1,aq^{n_1+1}/e_1,\dots,
e_m,aq^{n_m+1}/e_m,q^{-N};q)_k}{(q,\sqrt{a},-\sqrt{a},aq/e_1,e_1q^{-n_1},
\dots,aq/e_m,e_mq^{-n_m},aq^{N+1};q)_k}q^{(N-\nu)k},
\end{align*}
with $\nu=n_1+\cdots+n_m$.
Now the last sum vanishes by the special case of Gasper's summation stated in
\eqref{eq:vwp-km}.
\end{proof}

Using \cite[Lemma 2.1]{GSnew}, we can prove the following result
which is similar to \cite[Lemma 2.2]{GSnew}.
\begin{lemma}\label{lem:3}
Let $d,n$ be positive integers with $\gcd(d,n)=1$. Let $r$ be an
integer. Then
\begin{align*}
\sum_{k=0}^{m} [2dk+r]\frac{(q^r;q^d)_k^d}{(q^d;q^d)_k^d}q^{d(d-r-2)k/2} \equiv 0\pmod{[n]},   \\[5pt]
\sum_{k=0}^{n-1} [2dk+r]\frac{(q^r;q^d)_k^d}{(q^d;q^d)_k^d}q^{d(d-r-2)k/2} \equiv 0\pmod{[n]},
\end{align*}
where $0\leqslant m\leqslant n-1$ and $dm\equiv -r\pmod{n}$.
\end{lemma}

We have collected enough ingredients which enables us
to prove Theorem~\ref{thm:1}.

\begin{proof}[Proof of Theorem \ref{thm:1}]
The $q$-congruence \eqref{eq:thm1} modulo $[n]$ follows from Lemma \ref{lem:3} immediately.
In what follows, we shall prove the modulus $\Phi_n(q)^3$ case of \eqref{eq:thm1}.

For $M=(dn-n-r)/d$, the left-hand side of \eqref{eq:thm1} can be written as the following
multiple of a terminating $_{d+5}\phi_{d+4}$ series:
\begin{align*}
[r]
\sum_{k=0}^{(dn-n-r)/d}\frac{(q^r,q^{d+r/2},-q^{d+r/2},
q^r,\ldots,q^r,q^{(d+r)/2},q^{d+(d-1)n},q^{r-(d-1)n};q^d)_k}
{(q^d,q^{r/2},-q^{r/2},q^d,\ldots,q^d,q^{(d+r)/2},q^{r-(d-1)n},q^{d+(d-1)n};q^d)_k}
q^{d(d-r-2)k/2}.
\end{align*}
Here, the $q^r,\ldots,q^r$ in the numerator
means $d-1$ instances of $q^r$, and similarly, the $q^d,\ldots,q^d$
in the denominator means $d-1$ instances of $q^d$.
By Andrews' transformation \eqref{andrews},
we may rewrite the above expression as
\begin{align}
[r]\frac{(q^{d+r},q^{(r-d)/2-(d-1)n};q^d)_{(dn-n-r)/d}}
{(q^{(d+r)/2},q^{r-(d-1)n};q^d)_{(dn-n-r)/d}}
\sum_{j_1,\dots,j_{m-1}\geqslant 0}
\frac{(q^{d-r};q^d)_{j_1}\cdots(q^{d-r};q^d)_{j_{m-1}}}
{(q^d;q^d)_{j_1}\cdots(q^d;q^d)_{j_{m-1}}}& \notag\\[5pt]
\times\frac{(q^r,q^r;q^d)_{j_1}\dots (q^r,q^r;q^d)_{j_1+\dots+j_{m-2}}
 (q^{(d+r)/2},q^{d+(d-1)n};q^d)_{j_1+\dots+j_{m-1}}}
{(q^d,q^d;q^d)_{j_1}
\dots(q^d,q^d;q^d)_{j_1+\dots+j_{m-1}}}& \notag\\[5pt]
\times\frac{(q^{r-(d-1)n};q^d)_{j_1+\dots+j_{m-1}}}
{(q^{(3d+r)/2};q^d)_{j_1+\dots+j_{m-1}}}
q^{(d-r)(j_{m-2}+\dots+(m-2)j_1)+d(j_1+\dots+j_{m-1})}
&, \label{eq:multi}
\end{align}
where $m=(d+1)/2$.

It is easy to see that the $q$-shifted factorial $(q^{d+r};q^d)_{(dn-n-r)/d}$
contains the factor $1-q^{(d-1)n}$ which is a multiple of $1-q^n$.
Moreover, since none of $(r-d)/2$, $(d+r)/2$ and $(d+r)/2+dn-n-r-d$
are multiples of $n$, the $q$-shifted factorials
$$
(q^{(r-d)/2-(d-1)n};q^d)_{(dn-n-r)/d}
\quad \text{and}
\quad
(q^{(d+r)/2};q^d)_{(dn-n-r)/d}
$$
have the same number ($0$ or $1$) of factors of the form
$1-q^{\alpha n}$ ($\alpha\in\mathbb{Z}$).
Besides, the $q$-shifted factorial $(q^{r-(d-1)n};q^d)_{(dn-n-r)/d}$
is relatively prime to $\Phi_n(q)$. Thus we conclude that the fraction
before the multi-sum \eqref{eq:multi} is congruent to $0$ modulo $\Phi_n(q)$.

Note that the non-zero terms in the multi-summation in \eqref{eq:multi}
are those indexed by $(j_1,\ldots,j_{m-1})$ that satisfy the inequality
$j_1+\dots+j_{m-1}\leqslant (dn-n-r)/d$ because the
factor $(q^{r-(d-1)n};q^d)_{j_1+\dots+j_{m-1}}$ appears in the numerator.
None of the factors appearing in the denominator
of the multi-sum of \eqref{eq:multi} contain a factor of the form
$1-q^{\alpha n}$ (and are therefore relatively prime to $\Phi_n(q)$),
except for $(q^{(3d+r)/2};q^d)_{j_1+\dots+j_{m-1}}$ when
$$
(dn-d-n-r)/(2d)\leqslant j_1+\dots+j_{m-1}\leqslant (dn-n-r)/d.
$$
Since
$$
\frac{(q^{(d+r)/2};q^d)_{j_1+\dots+j_{m-1}}}
{(q^{(3d+r)/2};q^d)_{j_1+\dots+j_{m-1}}}
=\frac{1-q^{(d+r)/2}}{1-q^{(d+r)/2+(j_1+\dots+j_{m-1})d}},
$$
the denominator of the above fraction contains
a factor of the form $1-q^{\alpha n}$ if and only if
$j_1+\dots+j_{m-1}=(dn-d-n-r)/(2d)$ (in this case, the
denominator contains the factor $1-q^{(d-1)n/2}$).
Writing $n=ad-r$ (with $a\geqslant 1$),
we have $j_1+\dots+j_{m-1}=a(d-1)/2-(r+1)/2$.
Noticing that $m-1=(d-1)/2$ and $r\leqslant d-4$, there must exist an $i$
such that $j_i\geqslant a$.
Then $(q^{d-r};q^d)_{j_i}$ has the factor $1-q^{d-r+d(a-1)}=1-q^n$
which is divisible by $\Phi_n(q)$.
Hence the denominator of the reduced form of the multi-sum in
\eqref{eq:multi} is relatively prime to $\Phi_n(q)$.
It remains to show that the multi-sum in \eqref{eq:multi},
without the previous fraction, is congruent to $0$ modulo $\Phi_n(q)^2$.

By repeated applications of Lemma~\ref{lem:mod-square},
the multi-sum in \eqref{eq:multi}
(without the previous fraction),  modulo $\Phi_n(q)^2$, is congruent to
\begin{align*}
\sum_{j_1,\dots,j_{m-1}\geqslant 0}q^{(d-r)(j_{m-2}+\dots+(m-2)j_1)+d(j_1+\dots+j_{m-1})}
\frac{(q^{d-r};q^d)_{j_1}\cdots(q^{d-r};q^d)_{j_{m-1}}}
{(q^d;q^d)_{j_1}\cdots(q^d;q^d)_{j_{m-1}}}& \notag\\[5pt]
\times\frac{(q^{r+(m+1)n},q^{r-(m+1)n};q^d)_{j_1}\dots
(q^{r+(2m-2)n},q^{r-(2m-2)n};q^d)_{j_1+\dots+j_{m-2}}
}
{(q^{d-mn},q^{d+mn};q^d)_{j_1}
\dots(q^{d-(2m-3)n},q^{d+(2m-3)n};q^d)_{j_1+\dots+j_{m-2}}
}& \notag\\[5pt]
  \times \frac{(q^{d+(d-1)n},q^{(d+r)/2};q^d)_{j_1+\dots+j_{m-1}}
  (q^{r-(d-1)n};q^d)_{j_1+\dots+j_{m-1}}}
  {(q^{d-(2m-2)n},q^{d+(2m-2)n};q^d)_{j_1+\dots+j_{m-1}}
  (q^{(3d+r)/2};q^d)_{j_1+\dots+j_{m-1}}}
&,
\end{align*}
where $m=(d+1)/2$. However, this sum vanishes in light of the
$m=(d+1)/2$, $q\mapsto q^d$, $a= q^r$,
$e_1=q^{(d+r)/2}$, $e_m=q^{r-(2m-2)n}$, $e_i= q^{r-(m+i-2)n}$,
$n_1=(dn-d+n+r)/(2d)$,  $n_m=0$,
$n_i= (n+r-d)/d$, $2\leqslant i\leqslant m-1$, $N=(dn-n-r)/d$ case
of Lemma~\ref{lem:ms=0}. (It is easy to verify that
$N-n_1-\cdots-n_m=d(d-r-2)/2>0$.)
This proves that \eqref{eq:thm1} holds modulo $\Phi_n(q)^3$ for $M=(dn-n-r)/d$.

Since $(q^r;q^d)_k/(q^d;q^d)_k$ is congruent to $0$ modulo $\Phi_n(q)$
for $(dn-n-r)/d<k\leqslant n-1$,
we conclude that \eqref{eq:thm1} also holds modulo $\Phi_n(q)^3$ for $M=n-1$.
\end{proof}

%%%%%%%%%%%%%%%%%%%%%%%%%%%%%%%%%%%%%%%%%%%%%%%%%%%%%%%%%%%%%%%%%%%%%%%%%%%%%%%%%%%%%%%%%%%%%%%%%%%%%%%%%%%%%%%%%%%%%%%%%%%%%%%%%%%%%%%%%%%%%%%%%%%%%%%%%%%%%%%%%%%%%%%%%%%%%%%%%%%%%%%%%%%%%%%%%%%%%%%%%%%%%%%%%%%%%%%%%%%%%%%%%%%%%
\section{Proof of Theorem \ref{thm:2}}\label{sec:thm2}
We first give a simple lemma on a property of certain arithmetic progressions.
\begin{lemma}\label{lem:4}
Let $d$ and $r$ be odd integers satisfying $d\geqslant 3$, $r\leqslant d-4$
and $\gcd(d,r)=1$. Let $n$ be an integer such that $n\geqslant (d-r)/2$
and $n\equiv-r/2\pmod{d}$. Then there are no multiples of $n$ in the
arithmetic progression
\begin{equation}
  \frac{d+r}{2},\ \frac{d+r}{2}+d,\ldots, \frac{d+r}{2}+dn-2n-r-d.
  \label{eq:arithmetic}
\end{equation}
\end{lemma}

\begin{proof}
By the condition $\gcd(d,r)=1$, we have $\gcd((d+r)/2,(d-r)/2)=1$.
Suppose that
\begin{equation}
(d+r)/2+ad=bn  \label{eq:abn}
\end{equation}
for some integers $a$ and $b$ with $a\geqslant 0$.
Then $(d+r)/2+ad> (r-d)/2\geqslant -n$ and so $b\geqslant 0$.
Since $n\equiv (d-r)/2\pmod{d}$, we deduce from \eqref{eq:abn}
that $b\equiv -1\pmod{d}$ and thereby $b\geqslant d-1$.
But we have
$$
\frac{d+r}{2}+dn-2n-r-d
=dn-2n+\frac{d-r}{2}-d\leqslant (d-1)n-d,
$$
thus implying that no number in the arithmetic progression
\eqref{eq:arithmetic} is a multiple of $n$.
\end{proof}

\begin{proof}[Proof of Theorem \ref{thm:2}]
  As before, the $q$-congruence \eqref{eq:thm2} modulo $[n]$
  can be deduced from Lemma \ref{lem:3}.
It remains to prove the modulus $\Phi_n(q)^2$ case of \eqref{eq:thm2}.

For $M=(dn-2n-r)/d$, the left-hand side of \eqref{eq:thm2} can be written
as the following multiple of a terminating $_{d+5}\phi_{d+4}$ series
(this time we changed the position of $q^{(d+r)/2}$):
\begin{align*}
[r]
\sum_{k=0}^{(dn-2n-r)/d}\frac{(q^r,q^{d+r/2},-q^{d+r/2},q^{(d+r)/2},
q^r,\ldots,q^r,q^{d+(d-2)n},q^{r-(d-2)n};q^d)_k}
{(q^d,q^{r/2},-q^{r/2},q^{(d+r)/2},q^d,\ldots,q^d,q^{r-(d-2)n},q^{d+(d-2)n};q^d)_k}
q^{d(d-r-2)k/2}.
\end{align*}
Here, the $q^r,\ldots,q^r$ in the numerator
stands for $d-1$ instances of $q^r$, and similarly, the $q^d,\ldots,q^d$
in the denominator stands for $d-1$ instances of $q^d$.
By Andrews' transformation \eqref{andrews},
we may rewrite the above expression as
\begin{align}
[r]\frac{(q^{d+r},q^{-(d-2)n};q^d)_{(dn-2n-r)/d}}
{(q^d,q^{r-(d-2)n};q^d)_{(dn-2n-r)/d}}
\sum_{j_1,\dots,j_{m-1}\geqslant 0}
\frac{(q^{(d-r)/2};q^d)_{j_1}(q^{d-r};q^d)_{j_2}\cdots(q^{d-r};q^d)_{j_{m-1}}}
{(q^d;q^d)_{j_1}(q^d;q^d)_{j_2}\cdots(q^d;q^d)_{j_{m-1}}}& \notag\\[5pt]
\times\frac{(q^r,q^r;q^d)_{j_1}\dots (q^r,q^r;q^d)_{j_1+\dots+j_{m-2}}
 (q^r,q^{d+(d-2)n};q^d)_{j_1+\dots+j_{m-1}}}
{(q^{(d+r)/2},q^d;q^d)_{j_1}(q^d,q^d;q^d)_{j_1+j_2}
\dots(q^d,q^d;q^d)_{j_1+\dots+j_{m-1}}}& \notag\\[5pt]
\times\frac{(q^{r-(d-2)n};q^d)_{j_1+\dots+j_{m-1}}}
{(q^{d+r};q^d)_{j_1+\dots+j_{m-1}}}
q^{(d-r)(j_{m-2}+\dots+(m-2)j_1)+d(j_1+\dots+j_{m-1})}
&, \label{eq:multi-3}
\end{align}
where $m=(d+1)/2$.

It is easily seen that the $q$-shifted factorial $(q^{d+r};q^d)_{(dn-2n-r)/d}$
has the factor $1-q^{(d-2)n}$ which is a multiple of $1-q^n$.
Clearly, the $q$-shifted factorial $(q^{-(d-2)n};q^d)_{(dn-2n-r)/d}$
has the factor $1-q^{-(d-1)n}$ (again being a multiple of $1-q^n$)
since $(dn-2n-r)/d\geqslant 1$ holds according to the conditions
$d\geqslant 5$, $r\leqslant d-4$, and $n\geqslant (d-r)/2$.
This indicates that the $q$-factorial
$(q^{d+r},q^{-(d-2)n};q^d)_{(dn-2n-r)/d}$ in the numerator of the
fraction before the multi-sum in \eqref{eq:multi-3} is divisible by $\Phi_n(q)^2$.
Further, it is not difficult to see that the $q$-factorial
$(q^d,q^{r-(d-2)n};q^d)_{(dn-2n-r)/d}$
in the denominator is relatively prime to $\Phi_n(q)$.

Like the proof of Theorem \ref{thm:1}, the non-zero terms
in the multi-sum in \eqref{eq:multi-3}
are those indexed by $(j_1,\ldots,j_{m-1})$ satisfying the inequality
$j_1+\dots+j_{m-1}\leqslant (dn-2n-r)/d$ because of the appearance of the
factor $(q^{r-(d-2)n};q^d)_{j_1+\dots+j_{m-1}}$ in the numerator.
By Lemma \ref{lem:4}, the $q$-shifted factorial $(q^{(d+r)/2},q^d)_{j_1}$
in the denominator does not contain a factor of the form $1-q^{\alpha n}$ for
$j_1\leqslant (dn-2n-r)/d$ (and are therefore relatively prime to $\Phi_n(q)$).
In addition, none of the other factors appearing in the denominator
of the multi-sum of \eqref{eq:multi-3} contain a
factor of the form $1-q^{\alpha n}$,
except for $(q^{d+r};q^d)_{j_1+\dots+j_{m-1}}$
when $j_1+\dots+j_{m-1}=(dn-2n-r)/d$
(in this case the denominator contains the factor $1-q^{(d-2)n}$).

Letting $n=ad+(d-r)/2$ (with $a\geqslant 0$),
we get $j_1+\dots+j_{m-1}=a(d-2)+(d-r)/2-1$. If $j_1\geqslant a+1$,
then $(q^{(d-r)/2};q^d)_{j_1}$
contains the factor $1-q^{(d-r)/2+ad}=1-q^n$. If $j_1\leqslant a$,
then $j_2+\dots+j_{m-1}\geqslant a(d-3)+(d-r)/2-1$.
Since  $m-2=(d-3)/2$, $d\geqslant 5$, and $r\leqslant d-4$,
there must be an $i$ with $2\leqslant i\leqslant m-1$ and $j_i\geqslant 2a+1$.
Then $(q^{d-r};q^d)_{j_i}$ contains the factor $1-q^{d-r+2ad}=1-q^{2n}$
which is a multiple of $\Phi_n(q)$.
Therefore, the denominator of the reduced form of the multi-sum in
\eqref{eq:multi-3} is relatively prime to $\Phi_n(q)$.
\end{proof}


\begin{thebibliography}{99}

\bibitem{Andrews75}G.E. Andrews,
Problems and prospects for basic hypergeometric functions, in:
{\em Theory and Application for Basic Hypergeometric Functions},
R.A. Askey, ed., Math. Res. Center, Univ. Wisconsin, Publ. No. 35,
Academic Press, New York, 1975, pp. 191--224.

\bibitem{BB}J.M. Borwein and P.B. Borwein,
{\em Pi and the AGM}, volume 4 of Canadian Mathematical Society Series
of Monographs and Advanced Texts, John Wiley \& Sons, Inc., New York, 1998.

\bibitem{Gasper}G. Gasper,
Elementary derivations of summation and transformation formulas for $q$-series,
in {\em Special Functions, $q$-Series and Related Topics}
(M.E.H.~Ismail, D.R.~Masson and M.~Rahman, eds.),
Amer. Math. Soc., Providence, R.I.,
Fields Inst. Commun. \textbf{14} (1997), 55--70.

\bibitem{GR} G.~Gasper, M.~Rahman,
{\em Basic hypergeometric series}, second edition,
Encyclopedia of Mathematics and Its Applications \textbf{96},
Cambridge University Press, Cambridge, 2004.

\bibitem{Gorodetsky}O. Gorodetsky, $q$-Congruences, with applications
to supercongruences and the cyclic sieving phenomenon,
{\em Int. J. Number Theory} \textbf{15} (2019), 1919--1968.

\bibitem{Guo-rima}V.J.W. Guo,
  Proof of some $q$-supercongruences modulo the fourth power of a
  cyclotomic polynomial,
{\em Results Math.} \textbf{75} (2020), Art.~77.

\bibitem{Guo-c2}V.J.W. Guo,
Proof of a generalization of the (C.2) supercongruence of Van Hamme,
{\em Rev. R. Acad. Cienc. Exactas F\'is. Nat.,
  Ser. A Mat.} \textbf{115} (2021), Art.~45.

\bibitem{GS19fifth}
V.J.W. Guo and M.J. Schlosser,
Proof of a basic hypergeometric supercongruence modulo
the fifth power of a cyclotomic polynomial,
\emph{J.\ Difference Equ.\ Appl.\ }\textbf{25} (7) (2019), 921--929.

\bibitem{GS19}
V.J.W. Guo and M.J. Schlosser,
Some new $q$-congruences for truncated basic hypergeometric series,
\emph{Symmetry} \textbf{11} (2019), no.~2, Art.~268.

\bibitem{GS20}V.J.W. Guo and M.J. Schlosser,
Some new $q$-congruences for truncated basic hypergeometric series:
even powers,
{\em Results Math.} \textbf{75} (2020),  Art. 1.

\bibitem{GSnew}V.J.W. Guo and M.J. Schlosser,
A new family of $q$-supercongruences modulo the fourth power
of a cyclotomic polynomial,
{\em Results Math.} \textbf{75} (2020), Art. 155.

\bibitem{GS1}V.J.W. Guo and M.J. Schlosser,
A family of $q$-hypergeometric congruences modulo the fourth power
of a cyclotomic polynomial,
{\em Israel J. Math.} \textbf{240} (2020), 821--835.

\bibitem{GS}V.J.W. Guo and M.J. Schlosser,
Some $q$-supercongruences from transformation formulas
for basic hypergeometric series,
{\em Constr. Approx.}, in press; https://doi.org/10.1007/s00365-020-09524-z

\bibitem{GW}V.J.W. Guo and S.-D. Wang,
Some congruences involving fourth powers of central $q$-binomial coefficients,
\emph{Proc. Roy. Soc. Edinburgh Sect. A} \textbf{150} (3) (2020), 1127--1138.

\bibitem{GuoZu}V.J.W. Guo and W. Zudilin, A $q$-microscope for supercongruences,
{\em Adv. Math.} \textbf{346} (2019), 329--358.

\bibitem{GuoZu2}V.J.W. Guo and W. Zudilin,
Dwork-type supercongruences through a creative $q$-microscope,
{\em J. Combin. Theory, Ser. A} \textbf{178} (2021), Art. 105362.

\bibitem{LW}L. Li and S.-D. Wang,
Proof of a $q$-supercongruence conjectured by Guo and Schlosser,
{\em Rev. R. Acad. Cienc. Exactas F\'is. Nat.,
  Ser. A Mat.} \textbf{114} (2020), Art.~190.

\bibitem{Liu}J-C. Liu,
On a congruence involving $q$-Catalan numbers,
{\em C. R. Math. Acad. Sci. Paris} \textbf{358} (2020), 211--215.

\bibitem{LP}J.-C. Liu and F. Petrov,
Congruences on sums of $q$-binomial coefficients,
{\em Adv. Appl. Math.} \textbf{116} (2020), Art.~102003.

\bibitem{Long}L. Long,
Hypergeometric evaluation identities and supercongruences,
{\em Pacific J. Math.} \textbf{249} (2011), 405--418.

%\bibitem{LR}L. Long and R. Ramakrishna,
%Some supercongruences occurring in truncated hypergeometric series,
%{\em Adv. Math.} \textbf{290} (2016), 773--808.

\bibitem{NP2}H.-X. Ni and H. Pan,
Divisibility of some binomial sums,
{\em Acta Arith.} \textbf{194} (2020), 367--381.

\bibitem{OZ}R. Osburn and W. Zudilin,
On the (K.2) supercongruence of {V}an {H}amme,
{\em J. Math. Anal. Appl.} \textbf{433}  (2016),  706--711.

\bibitem{Ramanujan}S.~Ramanujan,
Modular equations and approximations to~$\pi$,
\emph{Quart. J. Math. Oxford Ser.}~(2) \textbf{45} (1914), 350--372.

\bibitem{RS}H. Rosengren and M.J. Schlosser,
On Warnaar's elliptic matrix inversion and Karlsson--Minton-type
elliptic hypergeometric series,
\emph{J. Comput. Appl. Math.} \textbf{178} (2005), 377--391.

\bibitem{Straub}A. Straub,
Supercongruences for polynomial analogs of the {A}p\'ery numbers,
{\em Proc. Amer. Math. Soc.} \textbf{147} (2019), 1023--1036.

\bibitem{Swisher}H. Swisher, On the supercongruence conjectures of van {H}amme,
\emph{Res. Math. Sci.} (2015) 2:18.

\bibitem{Hamme}L. Van Hamme,
Some conjectures concerning partial sums of generalized hypergeometric series,
in: {\em $p$-Adic Functional Analysis} (Nijmegen, 1996),
Lecture Notes in Pure and Appl. Math. \textbf{192},
Dekker, New York (1997), 223--236.

\bibitem{WY}X. Wang and M. Yue,
  A $q$-analogue of the (A.2) supercongruence of Van Hamme
  for any prime $p\equiv 3\pmod{4}$,
{\em Int. J. Number Theory} \textbf{16} (2020), 1325--1335.

\bibitem{WYu}X. Wang and M. Yu, Some new $q$-congruences on double sums,
  {\em Rev. R. Acad. Cienc. Exactas F\'is. Nat.,
    Ser. A Mat.} \textbf{115} (2021), Art.~9.

\bibitem{Zud2009}W. Zudilin,
Ramanujan-type supercongruences,
{\em J. Number Theory} \textbf{129} (2009), no. 8, 1848--1857.

\bibitem{Zu19}W. Zudilin,
Congruences for $q$-binomial coefficients,
{\em Ann. Combin.} \textbf{23} (2019), 1123--1135.

\end{thebibliography}
\end{document}